\documentclass[preprint,12pt]{elsarticle}

\usepackage{amssymb}
\usepackage{amsmath}
\usepackage{amsthm}

\usepackage{enumerate}
\usepackage{hyperref}

\newtheorem{theorem}{Theorem}[section]
\newtheorem{corollary}[theorem]{Corollary}
\newtheorem{lemma}[theorem]{Lemma}
\newtheorem{proposition}[theorem]{Proposition}

\theoremstyle{definition}
\newtheorem{remark}[theorem]{Remark}

\newcommand{\C}{\mathbb{C}} 
\newcommand{\R}{\mathbb{R}} 
\newcommand{\Z}{\mathbb{Z}} 
\newcommand{\Q}{\mathbb{Q}} 
\newcommand{\N}{\mathbb{N}} 
\newcommand{\M}{{\rm M}} 
\newcommand{\GL}{\rm GL} 
\newcommand{\Qp}{\mathcal{Q}} 

\newcommand{\abs}[1]{| #1 |} 

\newcommand{\set}[1]{\left\{#1\right\}} 

\newcommand{\mat}[4]{\begin{pmatrix}#1 & #2\\#3 & #4\end{pmatrix}} 

\newcommand{\inv}{^{-1}} 

\newcommand{\np}{_{n+1}}
\newcommand{\nm}{_{n-1}}

\newcommand{\an}{a_n}
\newcommand{\anp}{a\np}

\newcommand{\zn}{z_n}
\newcommand{\znp}{z\np}

\newcommand{\pn}{p_n}

\newcommand{\pnm}{p\nm}

\newcommand{\qn}{q_n}

\newcommand{\qnm}{q\nm}

\newcommand{\deln}{\delta_n}
\newcommand{\delnm}{\delta\nm}

\newcommand{\san}{\set{\an}}
\newcommand{\szn}{\set{\zn}}
\newcommand{\spn}{\set{\pn}}
\newcommand{\sqn}{\set{\qn}}
\newcommand{\sdeln}{\set{\deln}}

\begin{document}

\begin{frontmatter}





\title{Continued fraction expansions of complex numbers, Lagrange's theorem, and badly approximable numbers}


\author{S. G. Dani} 

\affiliation{organization={UM-DAE Centre for Excellence in Basic Sciences},
            addressline={University of Mumbai, Santacruz}, 
            city={Mumbai},
            postcode={400098}, 
            state={Maharashtra},
            country={India}}
            
     
            
\author{Ojas Sahasrabudhe} 

\affiliation{organization={Department of Mathematics, Indian Institute of Technology Bombay},
	addressline={Powai}, 
	city={Mumbai},
	postcode={400076}, 
	state={Maharashtra},
	country={India}}
	


\begin{abstract}
	This paper concerns extension of the classical Lagrange theorem, on the eventual periodicity of continued fraction expansions of quadratic surds, and the versions of it found in the literature in the case of complex numbers. In this respect, firstly, we adopt a more general notion of continued fraction expansions, in place of those arising from the nearest integer algorithms. Secondly, the issue is formulated in terms of zeros of quadratic and Hermitian forms, and a result is proved in terms of certain sequences of matrices associated with them, via continued fraction expansions. The result may be considered as a matrix analogue of Lagrange's theorem in the general framework. The unified approach leads to generalizations of the Lagrange theorem on one hand, and an extended version of a result of Hines (2019) on badly approximable complex numbers, on the other hand.
\end{abstract}




\end{frontmatter}


\newpage

\section{Introduction}\label{Intro}

A classical theorem of Lagrange asserts that a real number $\alpha$ is a quadratic irrational if and only if its sequence of partial quotients with respect to the simple continued fraction expansion is eventually periodic. The result was generalized by A.\,Hurwitz in \cite{Hur-A} to complex numbers, considering continued fraction expansions in terms of the Gaussian integers, using the nearest integer algorithm. There have been various generalizations of the ideas in recent years, involving other Euclidean subrings of $\C$ in place of the ring of Gaussian integers, and more general algorithms and procedures for producing continued fraction expansions; see in particular \cite{Lak}, \cite{Hen}, \cite{Hin}, \cite{DN}, \cite{D-gen} and~\cite{D-lazy}. 

In this paper we present a unified and more general approach to the issue. The generalization concerns, firstly, adopting the framework of general continued fraction expansions, as pursued in particular in \cite{DN} and \cite{D-gen}. Secondly, in the extended set up our main result (Theorem~\ref{finmat}) establishes a property of  certain sequences of matrices, which may be thought of as a matrix version of Lagrange's theorem. Theorem~\ref{finmat} enables us to deduce generalizations of results concerning zeroes of quadratic forms as well as of Hermitian forms. Apart from achieving a generalization to a broader framework, the approach thus presents a unified perspective for the two streams. Various examples are discussed in the context of the results on general continued fraction expansions of quadratic surds.

Application of Theorem~\ref{finmat} to quadratic forms yields substantially general versions of Lagrange's theorem for continued fraction expansions, with respect to more general algorithms as well non-algorithmic procedures; see Corollary~\ref{Lagrange2}. On the other hand, application of Theorem~\ref{finmat} to Hermitian forms enables us to generalize a result of R.\,Hines \cite{Hin}. In the case of real numbers the quadratic irrationals, which form a countable set, are the only known simple examples of badly approximable numbers (though the latter  form an uncountable set). The situation is different for complex numbers, and in fact, Hines \cite{Hin} showed that there exist circles in $\C$ all of whose points are badly approximable. Apart from generalizing Hines' result to the general context (Corollary~\ref{hbounded}), we furnish an explicitly defined class of circles in $\C$ all of whose points are badly approximable (Corollary \ref{badcircles}), and deduce from it that the set of well-approximable complex numbers is totally disconnected (Corollary~\ref{well-appr}). 

We begin, in the next section, with an introduction to the general framework for continued fraction expansions referred to above.

\section{Preliminaries}\label{Section:Prelim}
Let $\Gamma$ be a discrete Euclidean subring of $\C$ containing $1$, spanning $\C$; thus $\Gamma$ is one of $\Z[i]$ (Gaussian integers), $\Z[\sqrt{2} i]$, $\Z[\frac{1}{2} (1+\sqrt{3} i)]$ (Eisenstein integers), $\Z[\frac{1}{2} (1+\sqrt{7} i)]$ or $\Z[\frac{1}{2} (1+\sqrt{11} i)]$. We denote by $K$ the quotient subfield of $\Gamma$ in $\C$, and let $\C'=\C\setminus K$. For each $\Gamma$ we define $r_0(\Gamma)=\sup_{z\in \C}\inf_{\gamma\in \Gamma} |z-\gamma|$; we note that $r_0(\Gamma)<1$ for all $\Gamma$ as above. 

Let $\san_{n\geq 0}$ be a sequence in $\Gamma$. A sequence $\szn_{n\geq 0}$ in $\C$ is called an \emph{iteration sequence} corresponding to $\san$ if for all $n\geq 0$, $0<|\zn-\an|<1$ and $\znp=(\zn-\an)\inv$; we note that the conditions in particular imply that for all $n\geq 1$, $|\zn|>1$ and $\an\neq 0$.

Consider any $z\in \C'$. Then sequences $\san$, with iteration sequences $\szn$ associated with them can be constructed as follows:
As $\Gamma$ is a Euclidean subring spanning $\C$, and $z\in \C'$, there exist $a\in\Gamma$ such that $0<|z-a|<1$; we note that this is the reason for restricting $\Gamma$ to be a Euclidean subring spanning $\C$. There are multiple possibilities for this $a$, with the specific number depending on $\Gamma$ and the location of $z$, out of which we may pick any one as $a_0$, and put $z_1=(z-a_0)\inv$. Now $z_1\in\C'$ and hence we can similarly pick an $a_1$ such that $0<|z_1-a_1|<1$, and set $z_2=(z_1-a_1)\inv$. Continuing the procedure recursively in this way we get a pair of sequences $\san_{n\geq 0}$ and $\szn_{n\geq 0}$, with the latter as an iteration sequence for the former. In particular one may make the choice at each stage so that $\an=f(\zn)$, where $f:\C'\to \Gamma$ is a function such that $\abs{\zeta-f(\zeta)}<1$ for all $\zeta$; such a function is called an \emph{algorithm} for continued fraction expansions, and the pair of sequences $\san$ and $\szn$ we shall refer to as \emph{algorithmic continued fraction expansions}. One of the common algorithms consists of choosing $f(z)$ to be the element of $\Gamma$ which is nearest to $z$ (a suitable convention may be used to fix the choice when there are more than one nearest elements). 

To relate the above with the classical continued fractions we begin with some further notations. With any sequence $\san_{n\geq 0}$ in $\Gamma$ we associate a pair of sequences $\spn_{n\geq -1}$ and $\sqn_{n\geq -1}$ in $\Gamma$, called the \emph{associated $\Qp$-pair}, defined by 
$$ p_{-1}=1, \ \ p_0=a_0, \ \ p_{n} = a_{n} p_{n-1} + p_{n-2}, \text{ and}$$
$$ q_{-1}=0, \ \ q_0=1, \ \ q_{n} = a_{n} q_{n-1} + q_{n-2}, \text{ for } 
n\in \N.$$ 

\noindent We note that $\pn\qnm-\pnm\qn =(-1)^{n+1}$, for all $n\geq 0$, as may be proved inductively from the definitions as above. 

In general it can happen that $\qn=0$ for some $n\in \N$. However, for sequences $\san$ constructed with an iteration sequence, starting with a $z\in \C'$ as described above, it turns out that $\qn\neq 0$ for all $n\geq 1$ and 
$\{\frac{\pn}{\qn}\}$ converges to $z$, as $n\to \infty$; this was first proved in \cite{D-lazy} (where they were referred to as ``lazy'' continued fraction expansions) and an independent and somewhat more direct proof is given below (cf. Proposition~\ref{lazy1}\eqref{lazy1.1} and Proposition~\ref{qntoinfinity}).
The relations as above then signify that $z$ may be realized as $$z=a_0+\cfrac{1}{a_1+\cfrac{1}{a_2{+}_{\ddots}}}$$
in the spirit of classical continued fraction expansion; $\frac{\pn}{\qn}$ can be seen to equal truncated expressions from this, at the respective $n$'s, and the dots on the right hand side stand for the limit of the truncated expressions, with the equality of the two sides signifying association of the sequence $\san$ as continued fraction expansion of $z$; the sequence $\san$ in this case is called the \emph{sequence of partial quotients} corresponding to the continued fraction expansion. The ratios $\frac{p_n}{q_n}$ are called the \emph{convergents} corresponding to the sequence $\san$. We call the sequence $\sqn_{n\geq 0}$ the corresponding \emph{denominator sequence}.

Let $\san_{n\geq 0}$ be a sequence in $\Gamma$, with an iteration sequence $\szn$. We call $\sup_{n\geq 1} |\zn-\an|$ the \emph{radius} of the iteration sequence. We note that for $z\in\C'$ and any $1\geq r>r_0(\Gamma)$ (defined above) there exists an $a\in\Gamma$ such that $|z-a|<r$; thus the pairs $\san_{n\geq 0}$ and $\szn_{n\geq 0}$ can be constructed, with the latter an iteration sequence for the former, with radius at most $r$. We note that when the pair $\san$, $\szn$ is associated with an algorithm $f:\C'\to \Gamma$ then the radius of $\szn$ is at most $\sup_{\zeta \in \C'} \abs{\zeta -f(\zeta)}$.

\subsection{General properties of continued fraction expansions}

We begin by noting certain interrelations between various sequences associated with a continued fraction expansion, following the notations as above. 

\begin{proposition}\label{lazy1}
	Let $z\in \C'$, $\san_{n\geq 0}$ be a continued fraction expansion of $z$ with $\szn_{n\geq 0}$ the associated iteration sequence, and let $\spn_{n\geq -1}$, $\sqn_{n\geq -1}$ be the corresponding $\Qp$-pair. Then for all $n\geq 0$ the following statements hold: 
	\begin{enumerate}[(i)]
		\item \label{lazy1.1}
		$\qn z - \pn = (-1)^n (z_1\cdots\znp)\inv$; in particular $\qn\neq 0$;
		\item \label{lazy1.3}
		$(\znp\qn+\qnm) z = \znp\pn+\pnm$; 
		\item \label{lazy1.4}
		if $|\qnm|\leq|\qn|$ then $\abs{z-\frac{\pn}{\qn}} \leq \abs{\qn}^{-2}(|\znp| -\abs{\frac{\qnm}{\qn}})\inv$.
	\end{enumerate}
\end{proposition}

\begin{proof} 
	The first statement may be proved inductively, and the others can be deduced from it via simple steps. We omit the details; see \cite{D-gen}, Proposition~2.1 for an idea of the proofs. 
\end{proof}

\begin{proposition}\label{qntoinfinity} 
	Let the notations be as in Proposition~\ref{lazy1}. Then $|\qn|\to\infty$, and $\frac{\pn}{\qn}\to z$, as $n\to \infty$. 
\end{proposition}

\begin{proof}
	Suppose to the contrary that $\sqn$ has a convergent subsequence. Since $\Gamma$ is discrete there exists a sequence $\{n_k\}$ tending to infinity, and a $q\in \Gamma$ such that $q_{n_k}=q$ for all $k$. Then by Proposition~\ref{lazy1}\eqref{lazy1.1} there are only finitely many possibilities for $p_{n_k}$, and by passing to a subsequence we may assume that $p_{n_k}=p\in \Gamma$ for all $k$. Then for any $k$ we have $\abs{qz-p} = \abs{q_{n_{k+1}}z-p_{n_{k+1}}} = \abs{z_1\cdots z_{n_{k+1}+1}}\inv < \abs{z_1\cdots z_{n_{k}+1}}\inv = \abs{q_{n_{k}}z-p_{n_{k}}} = \abs{qz-p},$
	which is a contradiction. 
	Hence $|\qn|\to \infty$ and $n\to \infty$. As $|\zn|>1$ for all $n\geq 1$, it follows from this, together with Proposition~\ref{lazy1}\eqref{lazy1.1}, that $\frac{\pn}{\qn}\to z$. 
\end{proof}

\subsection{Sequence of relative errors and neat subsets}

In this subsection we introduce certain notions concerning continued fraction expansions and discuss certain simple properties that play a crucial role in our proofs. 

Let $\Gamma$, $K$ and $\C'=\C\setminus K$ be as before. Let $z\in \C'$ and $\san_{n\geq 0}$ be a continued fraction expansion of $z$ over $\Gamma$. Let $\szn_{n\geq 0} $ be the associated iteration sequence and $\spn_{n\geq -1}$, $\sqn_{n\geq -1}$ the corresponding $\Qp$-pair. 

The sequence $\deln$, $n\geq 0$ defined by 
$$\deln:=\qn(\qn z-\pn) =\qn^2\left(z-\dfrac{\pn}{\qn}\right)$$
is called the \emph{sequence of relative errors} corresponding to the expansion.

\begin{proposition}\label{theta1}
	With the notations as above, for $n\geq 1$, if $|\qnm| \leq |\qn|$, then the following statements hold:
	
	(i) $|\deln| \leq (|\znp|-1)\inv$ and 
	(ii) $|\delnm| \leq |\deln|+1$.
\end{proposition}

\begin{proof} 
	(i) This is a direct consequence of Proposition~\ref{lazy1}\eqref{lazy1.4} and the condition in the hypothesis.
	
	(ii) We have
	$$\qn^2\delnm=\qn^2(\qnm^2 z-\pnm\qnm)=\qnm^2\deln+\qnm\qn (\pn\qnm-\pnm\qn).$$ 
	Dividing by $\qn^2$, and recalling that $\pn\qnm-\pnm\qn=(-1)^{n+1}$, 
	we get $$\delnm=\deln\left(\dfrac{\qnm}{\qn}\right)^2+(-1)^{n+1} \left(\dfrac{\qnm}{\qn}\right).$$ 
	Thus when $|\qnm| \leq |\qn|$ we get $|\delnm| \leq |\deln|+1.$
\end{proof}

We say that a subset $N$ of $\N$ is \emph {neat} for the continued fraction expansion if 
$$ \sup_{n\in N}|\zn-\an|<1 ~ \hfill{\text{ and }} ~ |\qnm| \leq |\qn| \text{ for all } n\in N. $$

\begin{corollary}\label{theta2}
	Let $z\in \C'$, $\san_{n\geq 0}$ be a continued fraction expansion of $z$, with iteration sequence $\szn_{n\geq 0}$, and let $\sdeln_{n\geq 0}$ be the corresponding sequence of relative errors. Let $N \subseteq \N$ be a neat subset for the expansion. Then $\set{\deln \mid n\in N}$ and $\set{\delnm \mid n\in N}$ are bounded subsets of $\C$.
\end{corollary}

\begin{proof}
	Let $\szn_{n\geq 0}$ be the iteration sequence corresponding to $\san$ and $\sqn_{n\geq 0}$ be the associated denominator sequence. Since $N$ is neat, we have $M:=\sup_{n\in N}|\zn-\an|<1$. Thus for $n\in N$ we have $|\znp| =|\zn-\an|\inv > M\inv.$ Since $|\qnm|\leq |\qn|$, by Proposition~\ref{theta1} we have $|\deln| \leq (M\inv-1)\inv$ and $|\delnm|\leq (M\inv-1)\inv +1$. Thus $\set{\deln \mid n\in N}$ and $\set{\delnm \mid n\in N}$ are bounded subsets of $\C$.	
\end{proof}

\begin{proposition}\label{neatsubsets}
	Let $z\in \C'$ and $\san_{n\geq 0}$ be a continued fraction expansion of $z$ and $\szn_{n\geq 0}$ be the corresponding iteration sequence. Let $\sqn_{n\geq 0}$ be the associated denominator sequence. If either of the following conditions holds, then there exists an infinite neat subset $N$ for the expansion.
	\begin{enumerate}[(i)]
		\item \label{neatsubsets_sup} 
		$\limsup_{n\in\N} |\zn-\an|< 1$;
		\item \label{neatsubsets_mono} 
		$\abs{\qnm} \leq \abs{\qn}$ for all $n\in\N$.
	\end{enumerate} 
\end{proposition}

\begin{proof}
	\eqref{neatsubsets_sup} The condition implies that there exist $0<r<1$ and $n_0$ such that for all $n\geq n_0$, $|\zn-\an|<r$. 
	Let $N=\set{n\geq n_0 \mid \abs{\qnm} \leq \abs{\qn}}$. 
	Since $\abs{\qn}\to \infty$ as $n\to \infty$, $N$ is infinite and the choices show that it is a neat subset. 
	
	\eqref{neatsubsets_mono} When $\abs{\qnm} \leq \abs{\qn}$ for all $n\in \N$, by Proposition 3.6 of \cite{DN} (or Theorem~2.1 of \cite{D-lazy}) we have $\limsup_{n\in\N} |\zn|>1$. Hence there exists $\alpha >1$ such that the set $N=\set{n\in\N \mid |\znp|>\alpha}$ is infinite, and it is a neat subset since $\znp =(\zn-\an)\inv.$
\end{proof}

\begin{remark}\label{neatsubsets_funda}
	We note that condition \eqref{neatsubsets_sup} as in Proposition~\ref{neatsubsets} holds for continued fraction expansions arising from the nearest integer algorithm. 
	More generally, for a continued fraction sequence corresponding to an algorithm $f:\C'\to \Gamma$, the condition holds whenever the fundamental set of $f$, namely $\set{z-f(z)\mid z\in\C'}$, is contained in $B(0,r):=\set{z\in \C \mid |z|<r}$ for some $r<1$. Thus an infinite neat subset exists in these cases. Similarly, neat infinite subsets exist when the corresponding iteration sequence has radius $r<1$. 
\end{remark}

The examples in the preceding remark rely on using condition~\eqref{neatsubsets_sup} from Proposition~\ref{neatsubsets}. To get examples via use of condition~\eqref{neatsubsets_mono} of the proposition, one should know suitable criteria for it to hold. An analysis in this respect, when $\Gamma$ is either the ring of Gaussian integers or the ring or Eisenstein integers was carried out in \cite{D-gen}. We recall in particular the following criterion from \cite{D-gen}; see Corollary~\ref{Lagrange2} for an application relating to the Lagrange theorem. 

\begin{proposition}\label{mono} {\rm (\cite{D-gen}, Theorem~3.2)}
	Let $z\in \C'$ and $\san_{n\geq 0}$ be a continued fraction expansion of $z$ over $\Gamma$ with $\sqn_{n\geq 0}$ the associated denominator sequence. Suppose that, for all $n\geq 1$, $|\an|>1$ and the following holds: if $|\anp|=\sqrt{2}$ then $\abs{\an+\overline{\anp}}\geq 2$ and if $|\anp|=\sqrt{3}$ then $\abs{2\an+\overline{\anp}}\geq 3$. Then $|\qnm|<|\qn|$ for all $n\geq 1$. 
\end{proposition}

We also note the following. While the algorithmic examples obtained using condition~\eqref{neatsubsets_sup} as discussed in Remark~\ref{neatsubsets_funda} necessarily involve the fundamental set of the algorithm being contained in $B(0,r)$ for some $r<1$, a variety of examples are possible where it is not contained in $B(0,r)$ for any $r<1$, where the conclusion as in Proposition~\ref{neatsubsets} can be upheld by verifying the monotonicity of $\set{|\qn|}$, where $\sqn_{n\geq 0}$ is the corresponding denominator sequence; this holds, for example, for the sequences corresponding to the nearest even integer algorithm treated by Julius Hurwitz in \cite{Hur-J}; a more detailed analysis, with a more general class of examples in this respect may be found in our preprint \cite{DS}, and we shall not go into them here on account of the bulk of the technical work involved.

\section{Zeroes of integral binary forms and quadratic polynomials}\label{Section:Hines}

Let $\Gamma$, $K$ and $\C'=\C\setminus K$ be as before. By $\C^2$ we shall denote the space of $2$-rowed column vectors, viewed also as $2\times 1$ matrices, with entries in $\C$; for convenience we shall also write the elements as $(\xi,\eta)^t$ with $\xi, \eta \in \C$ (with the $t$ standing for transpose). Also we consider functions defined on $\C\times\C$ as functions on $\C^2$ and vice-versa with this notation. By $\sigma :\C\to \C$ we denote the map which is either the identity map or the complex conjugation map. For any $\xi \in \C$ we denote $\sigma(\xi)$ by $\xi^\sigma$. 

We also associate with the continued fraction expansion a sequence of matrices $\set{g_n}_{n\geq 0}$ defined by 
$$g_n:=\mat{\pn}{\pnm}{\qn}{\qnm} \text{ for all } n\geq 0,$$
where $\spn_{n\geq -1}$, $\sqn_{n\geq -1}$ are as before. As noted before, $\pn\qnm-\pnm\qn =(-1)^{n+1}$, for all $n\geq 0$; thus, each $g_n$ has determinant $\pm 1$. 

We denote by $\M(2,\C)$ the space of all $2\times 2$ matrices with entries in $\C$. For any discrete subring $\Gamma$ as above, and any $k\in \N$, we denote by $\M(2, \Gamma)$ and $\M(2, k\inv\Gamma)$ the subgroups consisting of elements $X$ whose entries are in $\Gamma$ and $k\inv \Gamma$ respectively. For any matrix $X$ we denote by $X^t$ the transpose of $X$. For $X\in \M(2,\C)$ we denote by $X^\sigma$ the matrix obtained by applying $\sigma$ to each entry of $X$. We say that $X\in \M(2,\C)$ is \emph{$\sigma$-symmetric} if $X^t=X^\sigma$; thus, $X$ is $\sigma$-symmetric if it is symmetric and $\sigma$ is the identity map or Hermitian symmetric and $\sigma$ is the complex conjugation map. 

For $X=\mat{A}{B}{C}{D}\in \M(2,\C)$, by the \emph{$\sigma$-form} corresponding to $X$ we mean the function $f:\C \times \C \to \C$ defined by, for all $\xi, \eta\in\C$,
$$f(\xi, \eta)=A \xi^{\sigma}\xi +B\xi^{\sigma}\eta +C\eta^{\sigma}\xi +D\eta^\sigma \eta.$$ 
We note that $f(\xi, \eta)$ is the entry of the $1\times 1$ matrix $(\xi, \eta)^\sigma X (\xi, \eta)^t$. 
By a \emph{nontrivial zero} of the $\sigma$-form $f$ we mean $(\xi,\eta)^t \in \C^2$ such that $\eta \neq 0$, $\xi/\eta \in \C'$ and $f(\xi, \eta)=0$. 

\subsection{A theorem for $\sigma$-symmetric forms}

For any $g\in \GL(2,\C)$ and $X\in \M(2,\C)$ let $X_{g,\sigma}=(g^t)^\sigma Xg$. We prove the following:

\begin{theorem}\label{finmat}
	Let $X\in\M(2,\C)$ be a $\sigma$-symmetric matrix, $f$ the corresponding $\sigma$-form and let $(\xi, \eta)^t\in \C^2$ be a nontrivial zero of $f$. 
	Let $\set{g_n}_{n\geq 0}$ be the sequence in $\GL(2,\Gamma)$ associated with a continued fraction expansion of $\xi/\eta$ and $N\subseteq\N$ be a neat infinite subset for the expansion. Then $\set{X_{g_n,\sigma} \mid n\in N}$ is a bounded subset in $\M(2,\C)$. Moreover, for any $X\in \M(2, k\inv \Gamma)$, where $k\in \N$, $\set{X_{g_n,\sigma} \mid n\in N}$ is finite. 
\end{theorem}

\begin{proof}
	Let $X= \mat{A}{B}{C}{D}\in \M(2,\C)$ be given; thus for any $\xi, \eta \in \C$ we have $f(\xi,\eta) =A \xi^{\sigma}\xi +B\xi^{\sigma}\eta +C\eta^{\sigma}\xi +D\eta^\sigma \eta$. For any $g\in \GL(2,\C)$, let $f^g:\C^2\to \C$ denote the function defined by $f^g((\xi,\eta)^t):=f(g(\xi,\eta)^t)$. Thus $f^g$ is the $\sigma$-form corresponding to $X_{g,\sigma}$. Let
	$$f^{g_n}(\xi, \eta)= A_n \xi^{\sigma}\xi + B_n\xi^{\sigma}\eta + C_n\eta^{\sigma}\xi + D_n\eta^\sigma \eta$$ for all $\xi, \eta$, where $A_n, B_n, C_n, D_n$ are the entries of $X_{g_n,\sigma}$, namely, $X_{g_n,\sigma}=\mat{A_n}{B_n}{C_n}{D_n}$. 
	\\Substituting for $g_n$ as $\mat{\pn}{\pnm}{\qn}{\qnm}$ 
	we see that 
	$$A_n=f(\pn,\qn) \text{ and } D_n=f(\pnm,\qnm)  \text{ for all }n\geq 0.$$
	Now let $z=\xi/\eta$ and $\deln=\qn^2(z-\frac{\pn}{\qn})$ be the sequence of relative errors for the continued fraction expansion as in the hypothesis. Thus we have 
	$\frac{\pn}{\qn}=z-\frac{\deln}{\qn^2}$ for all $n\geq 0$. Hence 
	$$A_n=f(\pn,\qn)=\qn^2\,f\left(\frac{\pn}{\qn}, 1\right)=\qn^2\,f\left(z-\frac{\deln}{\qn^2}, 1\right).$$
	On substituting in the expression for $f$, the values $\xi=z-\frac{\deln}{\qn^2}$ and $\eta =1$, the preceding equation yields
	$$\qn^{-2}A_n= A\left(z-\frac{\deln}{\qn^2}\right)^\sigma \left(z-\frac{\deln}{\qn^2}\right)+B\left(z-\frac{\deln}{\qn^2}\right)^\sigma +C\left(z-\frac{\deln}{\qn^2}\right)+D.$$ 
	As $f(z,1)=f(\xi,\eta)=0$, it follows that 
	$$\qn^{-2}A_n= A\left((\gamma_n^\sigma+\gamma_n)z+\gamma_n^\sigma\gamma_n\right)
	+ B\gamma_n^\sigma + C\gamma_n,$$ where $\gamma_n=-\deln/\qn^2$.
	We have $\abs{\gamma_n}=\abs{\deln}/\abs{\qn}^2$ and $\abs{\gamma_n^{\sigma}\gamma_n}\leq \abs{\deln}^2/\abs{\qn}^2$, as $\abs{\qn}\geq 1$. 
	Thus we get that 
	$$|A_n| \leq \left(|2Az|+|B|+|C|\right)|\deln|+|A|\,|\deln|^2, \text{ for all } n\geq 0.$$
	Also, 
	$$|D_n|= |f(\pnm,\qnm)| \leq (|2Az|+|B|+|C|)\,|\delnm|+|Az|\,|\delnm|^2,$$ for all $n\geq 0$.
	
	Now let $N$ be a neat subset of $\N$ for the continued fraction expansion as in the hypothesis. Then by Corollary~\ref{theta2} $\set{\deln \mid n\in N}$ and $\set{\delnm \mid n\in N}$ are bounded subsets of $\C$.
	Hence the above conclusions imply that $\set{A_n\mid n\in N}$ and $\set{D_n\mid n\in N}$ are bounded. 
	
	Now recall that by hypothesis $X$ is $\sigma$-symmetric. Therefore $X_{g_n,\sigma}$ is also $\sigma$-symmetric, and in particular $C_n= B_n^\sigma$ for all $n$ since $\det (g_n)=\pm 1$. Thus the determinants of $X$ and $X_{g_n,\sigma}$ are, respectively, $AD-BB^\sigma $ and $A_n D_n- B_n B_n^\sigma$, for all $n$. From the definition of $X_{g_n,\sigma}$ we see that $\det(X_{g_n,\sigma})=\det (X)$ for all $n$. Hence we get that $B_nB_n^\sigma =A_nD_n-AD +BB^\sigma $. 
	Since $\set{A_n \mid n\in N}$ and $\set{D_n \mid n\in N}$ are bounded, the preceding conclusion shows that for either choice of $\sigma$, the sequence $\set{|B_n|^2\mid n\in N}$ is bounded. Thus $\set{B_n\mid n\in N}$ is bounded, and since $C_n= B_n^\sigma$ for all $n$, it follows also that $\set{C_n\mid n\in N}$ is bounded. This completes the proof of the first assertion in the theorem. The second assertion is immediate from the first and the discreteness of $k\inv\Gamma$.
\end{proof}

\begin{remark}\label{thebound}
	A perusal of the proof of Theorem~\ref{finmat} shows, in the notation of the theorem, that all entries of $X_{g_n,\sigma}$, $n\in N$, are bounded by a computable constant depending only on $X$ and $\sup\,\set{|\deln|\mid n\in N}.$
\end{remark}

We shall next apply the theorem to study zeroes of $\sigma$-forms. Towards this we first note here the following, in the notation as above. 

\begin{lemma}\label{lem}
	Let $X\in\M(2,\C)$ and let $f$ be the corresponding $\sigma$-form. Let $z\in \C'$, $\szn_{n\geq 0}$ be the iteration sequence of a continued fraction expansion of $z$ and let $\set{g_n}_{n\geq 0}$ be the associated sequence of matrices in $\GL(2,\C)$. For $n\in \N$, let $f_n$ be the $\sigma $-form corresponding to $X_{g_n,\sigma}$. Then $f(z,1)=0$ if and only if $f_n(\znp,1)=0$.
\end{lemma}

\begin{proof}
	The definition of the $\sigma$-form $f$ corresponding to $X$ readily shows that $(\xi, \eta)^t$ is a zero of $f$ if and only if, for any $g\in \GL(2,\C)$, $g\inv(\xi, \eta)^t$ is a zero of the $\sigma$-form corresponding to ${X_{g,\sigma}}$. 
	On the other hand, for all $n\in \N$, $g_n(\znp, 1)^t= (\pn\znp+\pnm, \qn\znp+\qnm)^t$, and by Proposition~\ref{lazy1}\eqref{lazy1.3} it is a nonzero multiple of $(z,1)^t$. 
	Hence $(z,1)$ is a zero of $f$ if and only if $(\znp,1)$ is a zero of $f_n$. 
\end{proof}

\subsection{Application to roots of quadratic polynomials} 

Now let $\Gamma$ be a discrete subring of $\C$ and $K$ be its quotient field. Let $f(z)=az^2+bz+c$ be a quadratic polynomial, with $a,b, c\in \Gamma$ which is irreducible over $K$. It may be recalled that for a quadratic polynomials with coefficients in $\Z$ which is irreducible over $\Q$ the classical theorem of Lagrange asserts that the simple continued fraction expansions $\san$ of either of the roots is eventually periodic, namely there exist $k$ and $n_0$ such that $a_{n+k}=a_n$ for all $n\geq n_0$. Analogous results have been known in the framework as above, for complex quadratic polynomials, under various conditions on the continued fraction expansions. In the following corollary we establish a version of such a result for a larger class of continued fraction expansions, namely those admitting a neat subset. 

Let $\Gamma$ be a Euclidean subring and $K$ the quotient field. By a \emph{quadratic surd} with respect to $\Gamma$ we mean a root of a quadratic polynomial with coefficients in $\Gamma$ which is not contained in $K$. From Theorem~\ref{finmat} we get the following variant of Corollary~4.5 of \cite{DN}, for all $\Gamma$'s as above. 

\begin{corollary}\label{Lagrange}
	Let $z$ be a quadratic surd with respect to $\Gamma$. Let $\san_{n\geq 0}$ be a continued fraction expansion of $z$ and $\szn_{n\geq 0}$ the corresponding iteration sequence. Then for any neat subset $N$ for the expansion the sets $\set{\znp\mid n\in N}$ and $\set{\anp\mid n\in N}$ are finite. 
\end{corollary}

\begin{proof}
	Let $a,b,c\in\Gamma$ be such that $az^2+bz+c=0$. We shall apply Theorem~\ref{finmat} with $X=\mat{a}{\frac{1}{2} b}{\frac{1}{2} b}{c}$ and $\sigma$ the identity map. Then the corresponding $\sigma$-form is given by $f(\xi,\eta)=a\xi^2+b\xi\eta+c\eta^2$ for all $\xi, \eta \in \C$. 
	Since $z$ is a quadratic surd, $(z,1)$ is a nontrivial zero of $f$, and the theorem implies, in the notation as before, that $\set{{X_{g_n, Id}} \mid n\in N}$ is a finite set. Let $f_n$ denote the $\sigma$-form corresponding to $X_{g_n, Id}.$ Then we get that $\set{f_n(\zeta, 1) \mid n\in N}$ is a finite collection of polynomials in $\zeta$. By Lemma \ref{lem} each $\znp$, $n\in N$, is a root of one of the polynomials from this finite collection. It follows that $\set{\znp \mid n\in N}$ is finite. Since each $a_{n+1}$ is an element of $\Gamma$ within distance 1 from $z_{n+1}$ it follows that $\set{\anp\mid n\in N}$ is finite.
\end{proof}

It is easy to see that if $\san_{n\geq 0}$ is a continued fraction expansion of a number $z\in \C'$ and it is eventually periodic then $z$ is a quadratic surd over $K$. The main point of various versions of Lagrange's theorem concerns the converse. In this regard combining Corollary~\ref{Lagrange} with Proposition~\ref{neatsubsets} we deduce the following:

\begin{corollary}\label{Lagrange2}
	Let $z$ be a quadratic surd with respect to $\Gamma$.
	Let $\san_{n\geq 0}$ be a continued fraction expansion of $z$, $\szn_{n\geq 0}$ the corresponding iteration sequence and $\sqn_{n\geq 0}$ the associated denominator sequence. Suppose that either $\szn$ is of radius $r<1$ or $\set{|\qn|}$ is monotonic. Then we have the following:
	\begin{enumerate}[(i)]
		\item \label{Lagrange2.1} 
		 There exist $\zeta \in \C'$ and $a\in \Gamma$ such that $\zn=\zeta$ and $\an=a$ for infinitely many $n\in \N$; if, moreover, $\san$ is associated with an algorithm then $\san$ is eventually periodic. 
		\item \label{Lagrange2.2}
		If $\szn$ is of radius $r<1$ as well as the sequence $\set{|\qn|}$ is monotonic then the sets $\set{\zn\mid n\in \N}$ and $\set{\an\mid n\in \N}$ are finite. 
	\end{enumerate} 
\end{corollary}

\begin{proof} 
	\eqref{Lagrange2.1} Under the conditions as in the hypothesis, by Proposition~\ref{neatsubsets} there exists an infinite neat subset, say $N$. Then by Corollary~\ref{Lagrange} $\set{\znp \mid n\in N}$ is finite. It follows that there exists a $\zeta \in \C'$ such that $\zn=\zeta$ for infinitely many $n\in \N$. For all these $n$, $\an$ is an element of $\Gamma$ within distance 1, and hence there exists $a\in \Gamma$ such that $\an=a$ for infinitely many of them. This proves the first part of the statement in \eqref{Lagrange2.1}. 
	
	Now suppose that the continued fraction expansion is algorithmic. By the previous assertion there exist $n_0$ and $k\geq 1$ such that $z_{n_0+k}=z_{n_0}$. Since the continued fraction expansion follows an algorithm we get that $a_{n_0+k}=a_{n_0}$ and, in turn, that $z_{n_0+1+k}=z_{n_0+1}$. Recursively it follows that $a_{n+k}=a_n$ for all $n\geq n_0$, viz. that $\san$ is eventually periodic. 
	
	\eqref{Lagrange2.2} The given conditions mean that $\N$ is a neat subset for the expansion. The assertion therefore follows from Corollary~\ref{Lagrange}.
\end{proof}

\begin{remark}
	Let $z$ be a quadratic surd with respect to $\Gamma$ and $\san_{n\geq 0}$ be a general continued fraction expansion constructed as seen in \S\,\ref{Section:Prelim}. Then $\san$ need not be eventually periodic. This may be seen as follows. Let $z=\sqrt{2}$. Then we have the classical simple continued fraction expansion for $z$, with $a_0=1$ and $\an=2$ for all $n\geq 1$. The corresponding iteration sequence $\szn_{n\geq 0}$ in this case consists of $\zn=1+\sqrt{2}$ for all $n\geq 1$. We produce modified expansions as follows. Let $m\in \N$ and pick $a_m$ to be $3$. Then we get $z_{m+1}=(z_m-a_m)\inv=-1-\frac{1}{\sqrt{2}}$. We next choose $a_{m+1}=-2$ and get $z_{m+2}=(z_{m+1}-a_{m+1})\inv=2+\sqrt{2}$. Picking $a_{m+2}=3$ we get 
	$z_{m+3}=(z_{m+2}-a_{m+2})\inv=1+\sqrt{2}$. This shows that if in the sequence $\san$ with $a_0=1$ and $\an=2$ for all $n\in \N$ if we replace the triple $(2,2,2)$ at $(a_m, a_{m+1}, a_{m+2})$ by $(3,-2, 3)$ the resulting sequence is also a continued fraction expansion of $\sqrt{2}$. Making such a substitution at infinitely many places, say starting at $m=4^k$, $k=1,2, \dots$, we get a continued fraction expansion which is not eventually periodic. Corollary~\ref{Lagrange2} shows that though there may be no eventual periodicity, nevertheless in all such expansions there will be at least one entry which will be found repeated infinitely many times. 
\end{remark}

\begin{remark}
	Let $r_0=r_0(\Gamma)$, as defined in \S\,\ref{Section:Prelim}, and let $r_0<r<1$. Let $f:\C'\to \Gamma$ be the algorithm which assigns to each $z$ the farthest point in $\Gamma$ within distance $r$ from $z$ (with suitable tie-breaking convention when there are more thane one). Let $z$ be a quadratic surd and let $\san_{n\geq 0}$ be the continued fraction expansion with respect to the algorithm $f$. This would be different from the continued fraction expansion with respect to the nearest integer algorithm. However, like the latter, $\san$ will also be eventually periodic. Various other algorithms can be conceived in a similar fashion for which also the conclusion holds. 
	
\end{remark} 

\subsection{Application to roots of Hermitian quadratic polynomials}

We next apply Theorem \ref{finmat} and prove the following analogue for Hermitian quadratic polynomials, viz. functions of the form $H(z,1)$, where $H$ is a Hermitian binary form, showing that if $z\in \C'$ is a root of $H(z,1)$ then, under certain general conditions, the partial quotients $\san_{n\geq 0}$ corresponding to the continued fraction expansion of $z$ are bounded; this extends a result of Hines \cite{Hin}, where it is proved in the special case of the nearest integer algorithm. We begin by noting the following.

\begin{remark}\label{rem:zero} 
	Let $P(z)=az\bar{z}+b\bar{z}+\bar{b} z+c$ where $a,b,c\in \Gamma$ and $a,c\in \R$ be a Hermitian quadratic polynomial. If $a\neq 0$ then the set of roots of $P$ consists of the circle $\abs{z+\frac{b}{a}}^2=(|b|^2-ac)/a$ (note that $(|b|^2-ac)/a\in\Q$) if the latter is nonnegative (a degenerate circle with a single point if it is zero), and is empty if it is negative. If $a=0$ and $b\neq 0$ then the set of roots is an affine line in $\C$, containing a dense set of rational points. 
\end{remark}

\begin{corollary}\label{hbounded}
	Let $P(\zeta)=a\zeta \bar{\zeta}+b\bar{\zeta}+\bar{b}\zeta +c$, where $a,b,c\in \Gamma$ with $a, c\in \R$ be a Hermitian quadratic polynomial. Let $z\in \C'$ be a root of $P$, $\san_{n\geq 0}$ be a continued fraction expansion of $z$, $\szn_{n\geq 0}$ the associated iteration sequence, and $\spn_{n\geq -1}, \sqn_{n\geq -1}$ be the corresponding $\Qp$-pair. 
	Suppose that $\abs{\qnm}\leq\abs{\qn}$ for all $n\in\N$. Then at least one of the following holds:
	\begin{enumerate}[(i)]
		\item\label{hbounded1} there exists an increasing sequence $\set{n_k}$ in $\N$ such that ${P(p_{n_k}/q_{n_k})=0}$ for all $k$;
		\item\label{hbounded2} $\san$ is bounded.
	\end{enumerate}
	In particular if $P$ has no root in $K$, then $\san$ is bounded, and moreover, if $\sup_{n\in \N} |\deln|=M <\infty$, where $\sdeln_{n\geq 0}$ is the corresponding sequence of relative errors, then a bound for $\san$ can be given depending only on $a,b,c$ and $M$. 
\end{corollary}

\begin{proof}
	Let $X=\mat{a}{b}{\bar{b}}{c}$, where $a,b,c$ are the coefficients of $P$ as in the hypothesis. Let $\sigma$ denote the complex conjugation map. Then $X$ is $\sigma$-symmetric and the corresponding $\sigma$-form is given by $f(\xi,\eta)=a\bar{\xi} \xi+b\bar{\xi} \eta+\bar{b}\bar{\eta} \xi+c\bar{\eta} \eta$, for all $\xi, \eta \in \C$. 
	Thus $(z,1)$ is a zero of $f$ and it is nontrivial since $z\in\C'$. Let $\set{g_n}_{n\geq 0}$ be the sequence of matrices associated with the continued fraction expansion as in the hypothesis. 
	Consider any neat infinite subset $N$ for the expansion; by Proposition~\ref{neatsubsets}\,\eqref{neatsubsets_mono} such subsets exist. Then, in the notation as before, by Theorem \ref{finmat} we get that $\set{{X_{g_n, \sigma}} \mid n\in N}$ is a finite collection of matrices. For $n\in \N$, let $f_n$ denote the $\sigma$-form corresponding to $X_{g_n, \sigma}$ and let $\varphi_n$ be the Hermitian polynomial defined by $\varphi_n(\zeta)=f_n(\zeta ,1)$ for all $\zeta \in \C$. 
	Then $\set{\varphi_n\mid n\in N}$ is a finite collection of polynomials in $\zeta$, and by Lemma \ref{lem} each $\znp$, $n\in N$, is a root of $\varphi_n$. We shall use this observation to prove the assertion in the corollary. 
	
	Suppose that statement \eqref{hbounded1} does not hold. Then there exists $n'\in \N$ such that for $n\geq n'$, $P(\pn/\qn)\neq 0$.
	Since $\set{|\qn|}$ is monotonic, as seen in the proof of Proposition \ref{neatsubsets}\eqref{neatsubsets_mono}, there exists $\lambda>1$ such that the subset $N$ defined by $N=\set{n\in\N \mid n\geq n' \text{ and }|\znp|>\lambda}$
	is an infinite neat subset for the expansion; and thus in the notation as above $\set{f_n}_{n\in N}$ is a finite collection. 
	As recalled in the proof of Theorem~\ref{finmat}, in the expression for $f_n$ in terms of $\zeta$ and $\bar{\zeta}$, the coefficient of $z\bar{z}$ is $f(\pn,\qn)$. Hence for $n\geq n'$ the leading coefficient of $\varphi_n$ is nonzero and therefore by Remark~\ref{rem:zero}, the set of roots of $\varphi_n$, if nonempty, consists of a circle. Since any $\znp$, $n\in N$ is a root of $\varphi_n$ and $\set{\varphi_n\mid n\in N}$ is a finite collection, it follows that $\set{\znp\mid n\in N}$ is contained in a union of finitely many circles in $\C$ and hence form a bounded subset of $\C$. On the other hand, for $\znp$, for $n\notin N$, by the definition of $N$ we have either $n< n'$ or $|\znp|\leq \lambda$. Altogether, we get that $\set{\zn\mid n\in \N}$ is bounded. Since for all $n$, $|\an|\leq |\zn|+1$ it follows that $\san$ is bounded in $\C$, which shows that statement\,\eqref{hbounded2} holds; this proves the first assertion in the corollary. 
	
	Finally suppose that $\sup_{n\in \N}|\deln| =M<\infty$. We recall that for $n\in N$, $\znp$ is contained in the circle defined by $\varphi_n(\zeta)=0$, and hence $|\znp|$ admits a bound depending only on the coefficients of $f_n$, namely the entries of $X_{g_n, \sigma}$. Hence by Remark~\ref{thebound}, $\set{\znp \mid n\in N}$ admits a bound depending only on $a,b,c$ and $M$, and hence so does $\set{\zn \mid n\geq 0}$, as $\lambda$ is an absolute constant. Since $|\an|\leq |\zn|+1$ for all $n$, this proves the last statement in the corollary. 
\end{proof}

\begin{remark}\label{rem:hbounded}
	Corollary~\ref{hbounded} generalizes a result of Hines \cite{Hin} proved in the special case when $\Gamma$ is the ring of Gaussian integers, the continued fraction expansion is with respect to the nearest integer algorithm, and $P$ has no root in $K$. In general, the partial quotients in the continued fraction expansion of a $z\in \C'$ being bounded relates to $z$ being ``badly approximable" by elements of $K$. This connection was discussed in \cite{Hin} in the case of expansions with respect to the nearest integer algorithm, for the ring of Gaussian integers.
\end{remark}

\subsection{Circles of badly approximable numbers}
We recall that a complex number $z$ is said to be \emph{badly approximable}, with respect to a discrete subring $\Gamma$, if there exists a $\delta >0$ such that $\abs{z-\frac{p}{q}}\geq \delta /|q^2|$, for all $p,q\in \Gamma$, $q\neq 0$. Clearly if $z\in\C$ is badly approximable with respect to $\Gamma$ then $z\in\C'$. It is proved in \cite{Hin} that a number $z\in \C$ is badly approximable if and only if the partial quotients in its continued fraction expansion with respect to the nearest integer algorithm form a bounded sequence; we have proved an analogous result for a large class of continued fraction expansions, which we propose to publish separately. 

In \cite{Hin} it was deduced, from the case of Corollary~\ref{hbounded} for the nearest integer algorithm, that there exist circles in $\C$ all whose points are badly approximable with respect to $\Gamma$. We describe here an explicit class of such circles. For $\zeta \in \C$ and $r>0$ we shall denote by $C(\zeta ,r)$ the circle with centre at $\zeta$ and radius $r$, viz. $C(\zeta ,r)=\set{z \in \C \mid |z-\zeta|^2=r^2}$. We note that when $\zeta\in K$ and $r^2\in\Q$, the circle
$C(\zeta ,r)$ consists of zeros of Hermitian quadratic polynomial with coefficients in $\Gamma$. 

\begin{corollary}\label{badcircles}
	Let $\Gamma $ be a Euclidean subring of $\C$ and $K$ be the quotient field of~ $\Gamma$. Let $\zeta \in K$ and $r>0$ such that $r^2\in \Q$. Then the following statements are equivalent:
	
	\begin{enumerate}[(i)]
		\item \label{notbat}
		There exists $z\in C(\zeta,r)$ which is not badly approximable with respect to~$\Gamma$.
		\item \label{norms-ratio}
		If $r^2=\frac{m}{n}$ with $m,n\in \N$ coprime to each other, then $m, n$ are norms of some elements of~$\Gamma$. 
	\end{enumerate}
\end{corollary}

\begin{proof}
	Since $\zeta \in K$, and $r^2\in \Q$, $C(\zeta,r)$ is the set of zeroes of a Hermitian quadratic polynomial with coefficients in $\Gamma$. Hence by Corollary~\ref{hbounded} and Remark \ref{rem:hbounded} all the points of $C(\zeta,r)$ are badly approximable with respect to $\Gamma$ if and only if $C(\zeta,r)\cap K=\emptyset$. We further note that, as $\zeta \in K$, $C(\zeta,r)\cap K=\emptyset$ if and only if $C(0,r)\cap K=\emptyset $. If $r^2=m/n$, where $m$ and $n$ are norms of some elements of~$\Gamma$, then there exist $p, q\in \Gamma$ such that $r^2=|p|^2/|q|^2$ and hence $p/q\in C(0,r)$, so $C(0,r)\cap K$ is nonempty. To complete the proof it now suffices to show that if $C(0,r)\cap K$ is nonempty and $r^2=m/n$, where $m, n\in \N$ are coprime, then $m$ and $n$ are norms of some elements of~$\Gamma$. Suppose $C(0,r)\cap K$ is nonempty and let $p,q\in \Gamma$ be such that $p/q\in C(0,r)$. Then $p\bar{p}=r^2 q\bar{q}$. Let $m, n \in \N$ be coprime integers such that $r^2=m/n$. Then we have $np\bar{p} = mq\bar{q}$.
	Since $m$ and $n$ are coprime, by the unique factorization property of $\Gamma$ we get that $m$ and $n$ are equivalent, respectively, to $p\bar{p}$ and $q\bar{q}$ modulo units, and hence $m$ and $n$ are the norms of $p$ and $q$, respectively. 
\end{proof}

\begin{remark}\label{rem3.12}
	We recall here some congruence conditions on $n\in \N$ which ensure that $n$ is not a norm of any element of $\Gamma$, where $\Gamma$ is one of the rings as above. The norms of elements of $\Z[i]$ and $\Z[\sqrt{2} i]$ are of the form $k^2+l^2$ and $k^2+2l^2$ respectively, for some $k, l\in \N$, and it follows that if $n\equiv -1\pmod{8}$ then $n$ is not a norm of any element in $\Z[i]$ or $\Z[\sqrt{2} i]$. Now let $\Gamma=\Z[\frac{1}{2} (1+\sqrt{j}\,i])$, where $j=3,7$ or $11$. Then the norms of all elements of $\Gamma$ are integers of the form $\frac{1}{4} (k^2+jl^2)$ for some $k, l \in \N$. Thus if $n$ is a norm of such an element, then $4n=k^2+jl^2$. for some $k, l \in \N$, so $4n\equiv k^2 \pmod{m}$, and in turn this implies that $n \pmod{j}$ is a quadratic residue modulo $j$. 
	As $-1$ is not a square modulo $j$ for the above values, it follows that if $n\equiv -1\pmod{j}$ then $n$ is not a norm of any elements of $\Z[\frac{1}{2}(1+\sqrt{j}\,i)]$, for $j=3,7$ and $11$. 
	Thus if $n$ is congruent to $-1$ modulo $1848\,(=8\times 3 \times 7\times 11)$, then by Corollary~\ref{badcircles} the circles $|z|^2=n$ and $|z|^2=\frac{1}{n}$ consists entirely of badly approximable numbers for all the $ \Gamma$'s as above. 
\end{remark}

\begin{remark}\label{rem3.13}
	
	The construction as in Remark~\ref{rem3.12} also shows that if $\Gamma$ is one the Euclidean rings as above and $K$ is the corresponding quotient field, then for any two distinct points $z,w \in \C$ there exist $\zeta \in K$ and $r>0$ such that $z$ and $w$ are contained in distinct connected components of $\C\setminus C(\zeta, r)$ and all points of $C(\zeta, r)$ are badly approximable with respect to $\Gamma$.
	
\end{remark}

A number is said to be \emph{well-approximable}, with respect to $\Gamma$, if it is not badly approximable with respect to $\Gamma$. Corollary~\ref{badcircles} in particular implies the following. 

\begin{corollary}\label{well-appr}
	Let $\Gamma$ be a Euclidean subring of $\C$. Then the set of complex numbers which are well-approximable with respect to $\Gamma$ is totally disconnected. 
\end{corollary}

\begin{proof}
	The proof is immediate from Remark~\ref{rem3.13}. 
\end{proof}

\end{document}